\documentclass[11pt]{amsart}
\usepackage{graphicx}
\usepackage{amsfonts} 
\usepackage{amsmath, amscd, amssymb, yhmath}
\usepackage{latexsym}
\usepackage{mathrsfs}
\usepackage{url}

\def \R{\mathbb{R}}

\def \M{\mathcal{M}}

\def \id{\mathrm{id}}

\def \BDiff{\mathop{\mathrm{BDiff}}}

\def \r{\mathcal{R}}
\def \colim{\mathop{\mathrm{colim}}}
\def \co{\colon\thinspace}
\def \MTSO{\mbox{MTSO}}

\title{The weak b-principle: Mumford conjecture}
\author{Rustam Sadykov}
\address{\newline Department of Mathematics\\ CINVESTAV
\newline Av. Instituto Politecnico Nacional 2508
\newline  Col. San Pedro Zacatenco 
\newline Mexico, D.F. CP 07360}
\email{rstsdk@gmail.com}

\newtheorem{theorem}{Theorem}[section]
\newtheorem{lemma}[theorem]{Lemma}

\newtheorem{proposition}[theorem]{Proposition}
\newtheorem*{b-principle}{The b-principle}
\newtheorem*{H-principle}{H-principle}

\theoremstyle{remark}
\newtheorem{remark}[theorem]{Remark}

\theoremstyle{definition}
\newtheorem{definition}[theorem]{Definition}

\newtheorem{example}[theorem]{Example}

\subjclass[2010]{Primary: 55N20; Secondary: 53C23}

\begin{document}

\begin{abstract} 
In this note we introduce and study a new class of maps called oriented colored broken submersions. This is the simplest class of maps that satisfies a version of the b-principle and in dimension $2$ approximates the class of oriented submersions well in the sense that 
every oriented colored broken submersion of dimension $2$ to a closed simply connected manifold is bordant to a submersion. 

We show that the Madsen-Weiss theorem (the standard Mumford Conjecture) fits  a general setting of the b-principle. Namely, a version of the b-principle for 
oriented colored broken submersions  together with the Harer stability theorem and Miller-Morita theorem implies the Madsen-Weiss theorem. 
\end{abstract}
\maketitle

\date{\today}

\section{Introduction}

A smooth map of manifolds $f\colon M\to N$ is said to be an \emph{immersion} if its differential is a fiberwise monomorphism $TM\to TN$ of tangent bundles. According to a remarkable theorem by Smale and Hirsch the space of immersions $M\to N$ of given manifolds with $\dim M<\dim N$ is weakly homotopy equivalent to a simpler topological space of \emph{formal immersions}, i.e., fiberwise monomorphisms $TM\to TN$. The Smale-Hirsch theorem was one of the primary motivations for the general Gromov \emph{h-principle}: given a differential relation, the space of its solutions is weakly homotopy equivalent to the space of its formal solutions \cite{Gr}.   

In \cite{Sa} (for a short review, see \cite{Sa1}) I proposed a stable homotopy version of the h-principle, the b-principle, motivated by a series of earlier results including \cite{An, Au3, El, GMTW, MW, RS, Sa3, Sa4, Sz, We}.  Namely, with every  open stable differential relation $\r$, there are associated a moduli space $\M_\r$ of solutions, a moduli space $h\M_\r$ of stable formal solutions, and a map $\alpha\co \M_\r\to h\M_\r$.  It turns out that $\M_\r$ is an H-space with a coherent operation, while $h\M_\r$ is an infinite loop space \cite{Sa}, whose stable homotopy type is relatively simple. The b-principle is the following conjecture. 
\begin{b-principle}
The canonical map $\M_\r\to h\M_\r$ is a group completion.
\end{b-principle}
When holds true, the b-principle allows us to perform explicit computations of invariants of  solutions. On the other hand,  the b-principle is true for most of the differential relations (see \cite{Sa} and references above); notable exceptions are the differential relations of oriented submersions of positive dimensions $d$. In this important exceptional case the b-principle inclusion coincides with the Madsen-Tillmann map 
\[
\alpha\co \sqcup \BDiff M\to \Omega^{\infty}\MTSO(d),
\]
where $\sqcup \BDiff M$ is the disjoint union of the classifying spaces of orientation preserving diffeomorphism groups of oriented closed (possibly not path connected) manifolds $M$ of dimension $d$, while $\Omega^{\infty}\MTSO(d)$ is the infinite loop space of the Madsen-Tillmann spectrum~\cite{GMTW}. 
The standard Mumford conjecture asserts that for $d=2$ and for a closed oriented surface $F_g$ of genus $g$, the map $\alpha|\BDiff F_g$ induces an isomorphism of rational cohomology rings in stable range of dimensions $*\ll g$.  The Mumford conjecture was proved in the positive by Madsen and Weiss in \cite{MW}, and later several other proofs of the Madsen-Weiss theorem were given in \cite{GMTW, EGM, GRW, Ha}. 

\begin{theorem}[Madsen-Weiss]\label{th:main1} The rational cohomology ring of $\BDiff F_g$ is a polynomial ring in terms of Miller-Morita-Mumford classes $\kappa_i$:
\[
\ \qquad   \qquad  \qquad \qquad H^*(\BDiff F_g; \mathbb{Q}) \simeq \mathbb{Q}[\kappa_1, \kappa_2, ...],  \qquad \qquad  \qquad \textrm{for } *\ll g,
\]
or, equivalently, the map $\alpha|\mathop\mathrm{BDiff} F_g$ is a rational homology equivalence in a stable range of dimensions. 
\end{theorem}

\begin{remark} In fact, Madsen and Weiss proved a stronger statement, which, in particular, implies that the map $\alpha|\mathop\mathrm{BDiff} F_g$ is an integral homology equivalence in a stable range.
\end{remark}

In the current note we study a new class of flexible maps---the class of colored broken submersions---that provides a good approximation to the class of submersions, retains the sheaf property, and  satisfies a version of the b-principle. More generally, we define colored broken solutions to an open stable differential relation; these enjoy many interesting properties including the following ones. 

\begin{itemize}
\item For an open stable differential relation $\mathcal{R}$ that \emph{does not}  satisfy the b-principle, a stable formal solution of $\mathcal{R}$ can be integrated into a broken solution (Theorem~\ref{th:2}). Thus, stable formal solutions differ from solutions only in broken components of the corresponding broken solutions. 
\item The pullback of a colored broken solution with respect to a generic smooth map is a colored broken solution. Thus, colored broken solutions form a class and therefore possess a moduli space (\S\ref{s:5a}).
\item The class of colored broken solutions satisfies the sheaf property, and therefore it is suitable for study by means of homotopy theory.  
\item Colored broken solutions of an open stable differential relation $\mathcal{R}$ satisfy a weak b-principle (Theorem~\ref{th:2}) even if solutions of $\mathcal{R}$ do not. 
\end{itemize}


To begin with we introduce the broken submersions/solutions in section \S\ref{s:2}. In sections \S\ref{s:4c}-\ref{sec:5} we recall the notions of a concordance and bordism. Next we show that the class of broken submersions approximates well the class of submersions (\S\ref{s:4a}-\S\ref{s:3}); in \S\ref{s:3} we essentially prove Theorem~\ref{th:4.2} (a complete proof is given in \S\ref{s:7c}). 

\begin{theorem}\label{th:4.2} 
Let $f\co M\to N$ be an oriented broken submersion of dimension $2$ to a simply connected manifold $N$. Suppose that the image of broken components of $f$ in $N$ is disjoint from $\partial N$; in particular, over $\partial N$ the map $f$ is a fiber bundle with fiber $F_g$. Suppose that $g\gg \dim N$.
Then $f$ is bordant to a fiber bundle by bordism which is a broken submersion itself. 
\end{theorem}

Theorem~\ref{th:4.2} relies heavily on the Harer stability theorem, and its proof is very much in spirit of a singularity theoretic argument by Eliashberg, Galatius and Mishachev in \cite{EGM}. 

 Next  we review the weak b-principle (\S\ref{s:5a}), and introduce the colored broken submersions (\S\ref{s:5}). The moduli space $\M_b$ of colored broken submersions is an H-space with coherent operation.  Its classifying space $B\M_b$ is known; it has essentially been determined in \cite{GMTW} (for a proof in present terms, see \cite{Sa}). 
Finally, in section~\ref{s:7c} we show that in view of the Harer stability theorem and the Miller-Morita theorem, the Madsen-Weiss theorem follows from the weak b-principle for colored broken submersions.

Colored broken submersions  are similar to (but have better properties than) marked fold maps. In particular, the moduli space of colored broken submersions of dimension $d$ is an appropriate homotopy colimit of classifying spaces $\mathop\mathrm{BDiff} M$ of diffeomorphism groups of manifolds of dimension $d$ with certain boundary components, compare with the original paper \cite{MW}. Colored broken submersions should be compared with enriched fold maps from \cite{EGM} of Galatius-Eliashberg-Michachev who used them to give a topological proof of the Madsen-Weiss theorem.  Note, however, that in contrast to enriched fold maps, colored broken submersions behave well with respect to taking pullbacks and possess a moduli space (\S\ref{s:5a}). We adopt much of the singularity theory technique from \cite{EGM}, but we do not use the major authors' tool: the Wrinkling theorem.  The determination of the classifying space $B\mathcal{M}_b$ is essentially from \cite{GMTW} (however, the rest of their proof of the Madsen-Weiss theorem is not necessary in the current setting). 

\subsection*{Acknowledgement} I am grateful to Soren Galatius for his generous help; the key idea to use $\mathcal{I}$-spaces (e.g., see \cite{Sch}) to link the b-principle to the Mumford conjecture is his. I would also like to thank Ivan Mart\'{i}n Protoss for presenting the material of the note in a series of talks in Topology Seminar in CINVESTAV. This paper was partially written while I was staying at the Max Planck Institute for Mathematics.

\section{Broken solutions}\label{s:2}

Given a smooth map $f\co M\to N$, a point $x\in M$ is said to be \emph{regular} if in a neighborhood $U$ of $x$ the map $f|U$ is a submersion. 
A point $x\in M$ is a \emph{fold point} if there are coordinate charts about $x$ and $f(x)$ such that 
\begin{equation}\label{eq:1}
      f(x_1,..., x_m) = (x_1,..., x_{n-1}, \pm x_n^2 \pm x_{n+1}^2 \pm \cdots \pm x_m^2),
\end{equation}
where $n$ is the dimension of $N$, and $x_1,..., x_m$ are coordinates in the coordinate chart about $x$. If every point in $M$ is regular or fold, then $f$ is said to be a \emph{fold map}.  
 It immediately follows from the local coordinate representation (\ref{eq:1}) of $f$ that the set of fold points of $f$ is a submanifold of $M$ of codimension $d+1$ where $d=\dim M-\dim N$.
 
\begin{figure}[htb]
\centering
\includegraphics[height=1.4in]{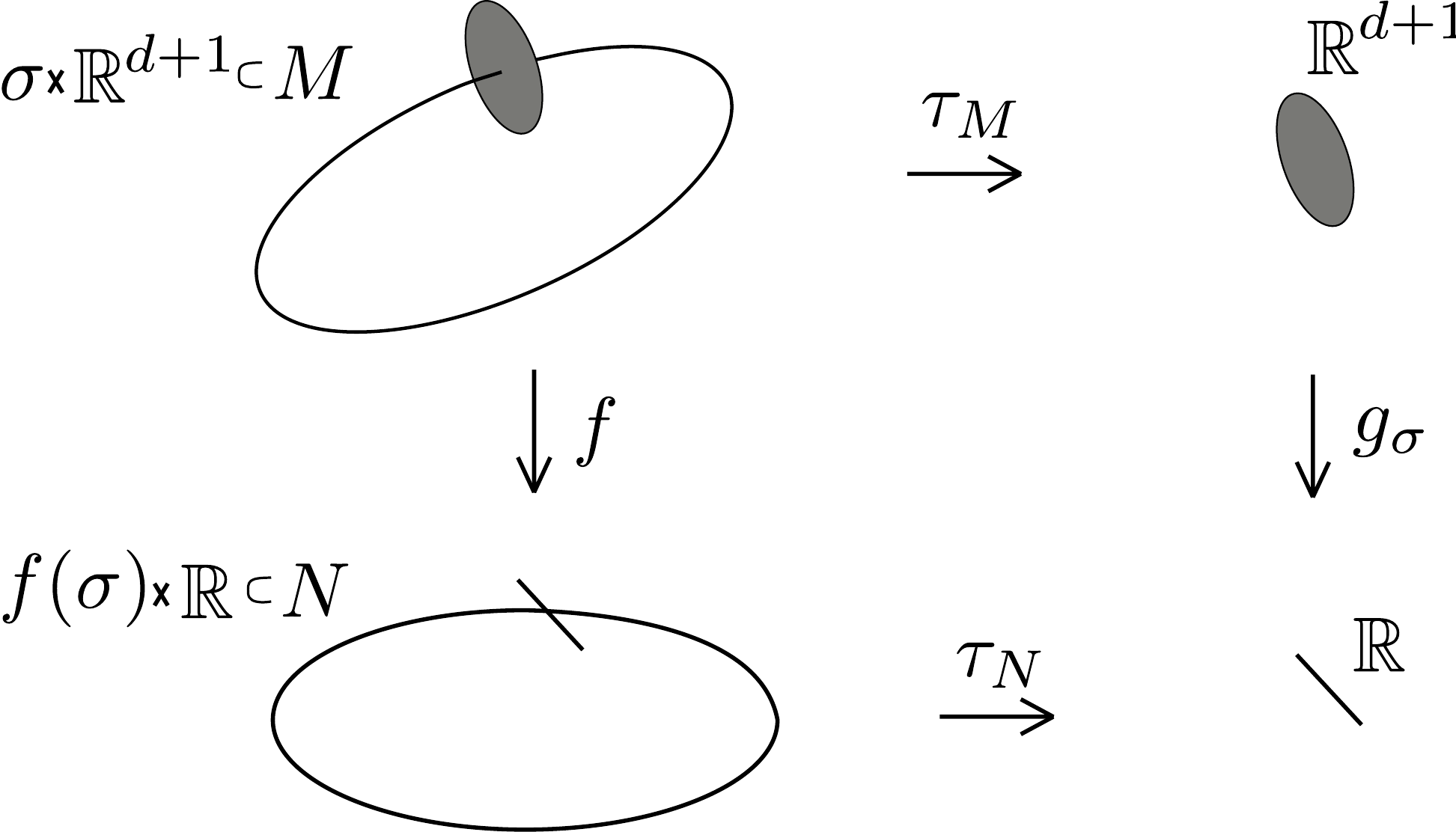}
\caption{A breaking component.}
\end{figure}

 Suppose that a path component $\sigma$ of fold points of $f$
is closed in $M$ and the restriction $f|\sigma$ is an embedding. 
Suppose that there is a submersion $\tau_M$ of a neighborhood of $\sigma$ in $M$ onto a neighborhood of $0$ in $\R^{d+1}$ such that the inverse image of $0$ is precisely $\sigma$. Then the map $\tau_M$ \emph{trivializes} the normal bundle of $\sigma$, though we do not fix a diffeomorphism of a neighborhood of $\sigma$ onto $\sigma\times \R^{d+1}$. Similarly, suppose that there is a map $\tau_N$ of a neighborhood of $f(\sigma)$ to $\R$ trivializing the normal bundle of $f(\sigma)$ in such a way  that $\tau_N\circ f= g_{\sigma}\circ \tau_M$ on the common domain, where $g_\sigma$ is a Morse function on $\R^{d+1}$ with one critical point. 
Then we say that $\sigma$ is a broken component;  the maps $\tau_M$ and $\tau_N$ are parts of the structure of a broken component. The minimum of the indices of the critical points of $g_\sigma$ and $-g_\sigma$ is called the \emph{index} of $\sigma$. 

\begin{remark}
The normal bundle in $M$ of a component $\sigma$ of fold points of a general fold map $f$ is not trivial, and $f|\sigma$ is not necessarily an embedding. Therefore not every component of fold points of a fold map admits a structure of a broken component. In fact, even if $f|\sigma$ is an embedding and the normal bundles of $\sigma$ in $M$ and $f(\sigma)$ in $N$ are trivial, the component $\sigma$ may still not admit a structure of a breaking component since $f$ near $\sigma$ may be twisted. 
\end{remark}

\begin{remark} Broken components of index $0$ are not compatible with certain nice structures including the structure of broken Lefschetz fibrations in the case of maps of $4$-manifolds into surfaces. For this reason in the general setting in \cite{Sa1} we prohibited broken components of index $0$  and proved the weak b-principle in the form of Theorem~\ref{th:2} with a less restrictive assumption of indices $\ne 0$. For the argument in the present paper, however, it is convenient to allow broken components of index $0$ (so that the space $\M_b$ in \S\ref{s:5} is connected). 
\end{remark}

Given an open stable  differential relation $\r$ imposed on maps of dimension $d$, suppose a map $f$ away of the broken fold components is a solution. Then we say that $f$ is a \emph{broken solution} of $\r$.  

\section{Bordisms}\label{s:4c}

We need the notion of an oriented bordism of maps of manifolds with boundaries. An \emph{oriented bordism} of a manifold with boundary is an oriented bordism with support in the interior of the manifold. An oriented bordism of maps is defined appropriately. 

\begin{definition} 
Let $M$ be an oriented compact manifold with corners  such that $\partial M$ is the union of $-M_0, M_1$ and $\partial M_0\times [0,1]$ where $\partial M_0\times \{i\}$ and $\partial M_i$  are identified for $i=0,1$, see Figure~\ref{fig:4}. In particular, the manifolds $\partial M_0$ and $\partial M_1$ are canonically diffeomorphic. The corners of $M$ are along $\partial M_0\times\{i\}$. 
\begin{figure}[h]
 \centering
\includegraphics[height=1.8in]{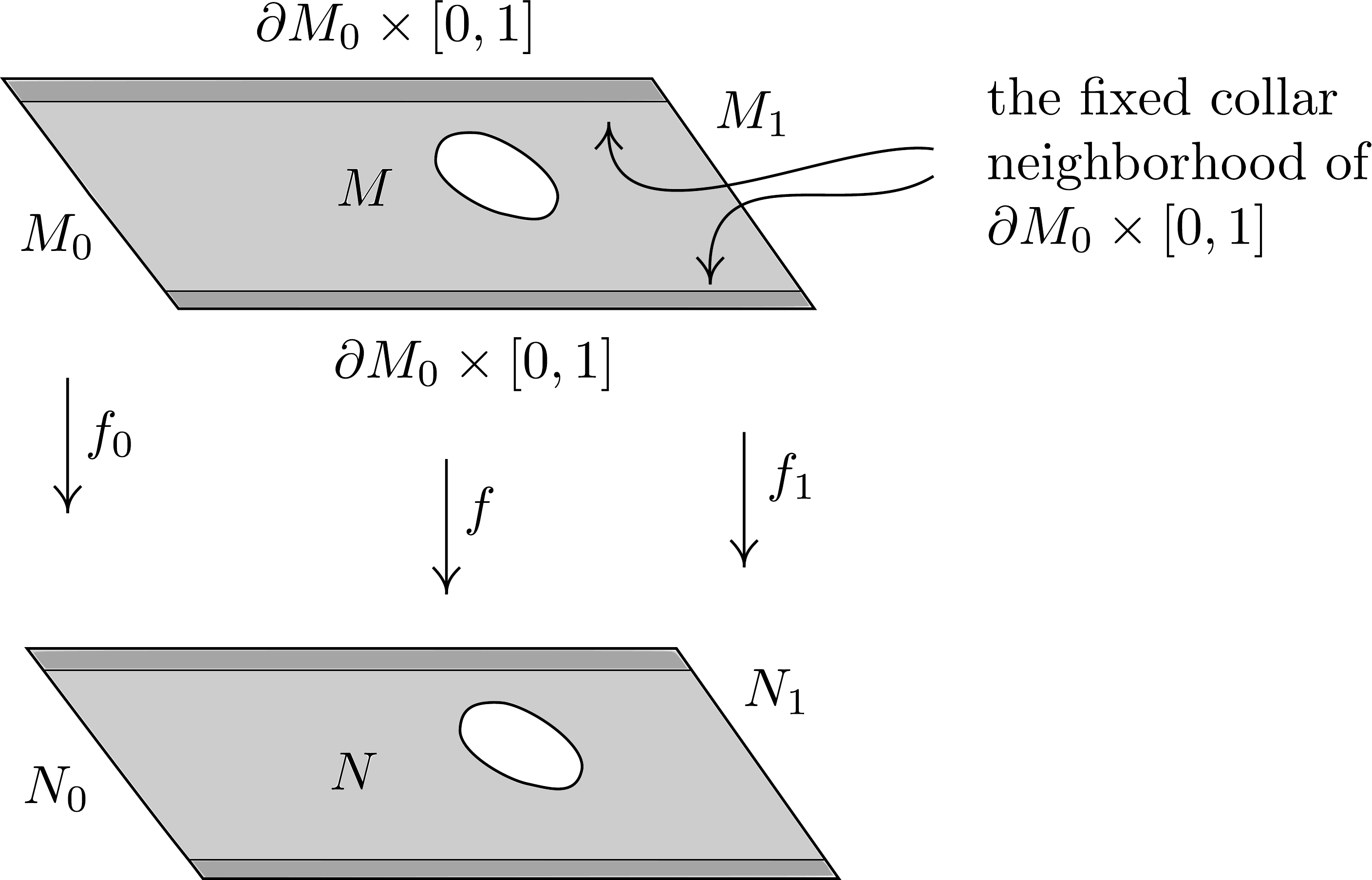}
\caption{An oriented bordism.}
\label{fig:4}
\end{figure}
Let $N$ be an oriented compact manifold with corners and with a similar decomposition of the boundary. Let $f\co M\to N$ be a map that preserves the decompositions. In particular, appropriate restrictions of $f$ define two maps \[
f_i\co (M_i, \partial M_i)\to (N_i, \partial N_i),  \qquad \mathrm{where}\quad i=0,1.
\]  We say that $f$ is an  \emph{oriented bordism} from $f_0$ to $f_1$ if  $f=f_i\times\id$ and $f=f_0\times \id_{[0,1]}$ over collar neighborhoods of  $M_i$ and 
 $\partial M_0\times [0,1]$ respectively.   
If $f_0, f_1$ belong to some class of maps, then we require that $f$ belongs to the same class. For example, a bordism of fiber bundles is a fiber bundle. 
\end{definition}
 
The product map $F_0\times \id_{[0,1]}\co M_0\times [0,1]\to N_0\times [0,1]$ is said to be a \emph{trivial} bordism. Let $m_0\subset M_0$ be a compact submanifold of codimension zero, and $f\co m\to N_0\times [0,1]$ a bordism of $f_0=F_0|m_0$. Then, there is a well-defined bordism $F\co M\to N_0\times [0,1]$ where $M$ is obtained from $M_0\times [0,1]$ by removing $m_0\times [0,1]$ and attaching $m$ along the new fiberwise boundary. The map $F$ coincides with $f$ over $m$ and with $F_0\times \id_{[0,1]}$ over the complement to $m$. We say that $F$ is  a bordism of $F_0$ with \emph{support} in $m_0$ and with core $f$.

\section{Concordances}\label{s:3c}

A bordism $M\to N$ of maps is said to be a \emph{concordance} if  the manifold $N$ is a product $N_0\times [0,1]$, and the decomposition of the boundary is the obvious one with $N_1=N_0\times \{1\}$. 
Thus, for example, two proper maps $f_i\co M_i\to N$ with $i=0,1$ of manifolds with empty boundaries are said to be \emph{concordant} if there is a proper map $f\co M\to N\times [0, 1]$ together with diffeomorphisms 
\[
f^{-1}(N\times [0, \varepsilon))\approx M_0\times [0, \varepsilon), \qquad  
f^{-1}(N\times (1-\varepsilon, 1])\approx M_1\times (1-\varepsilon, 1]
\]
onto collar $\varepsilon$-neighborhoods of $M_0$ and $M_1$ for some $\varepsilon>0$ such that  in view of these identifications  
\[
f|f^{-1}(N\times [0, \varepsilon)) = f_0\times \id_{[0, \varepsilon)}, \qquad  
f| f^{-1}(N\times (1-\varepsilon, 1])= f_1 \times  \id_{(1-\varepsilon, 1]},
\]   
see Figure~\ref{fig:8}. 
A concordance of maps of a given type is required to be a map of the same type, e.g., a concordance of submersions is a submersion. 
 
\begin{figure}[h]
\centering
\begin{minipage}{.5\textwidth}
\centering
\includegraphics[height=1.4in]{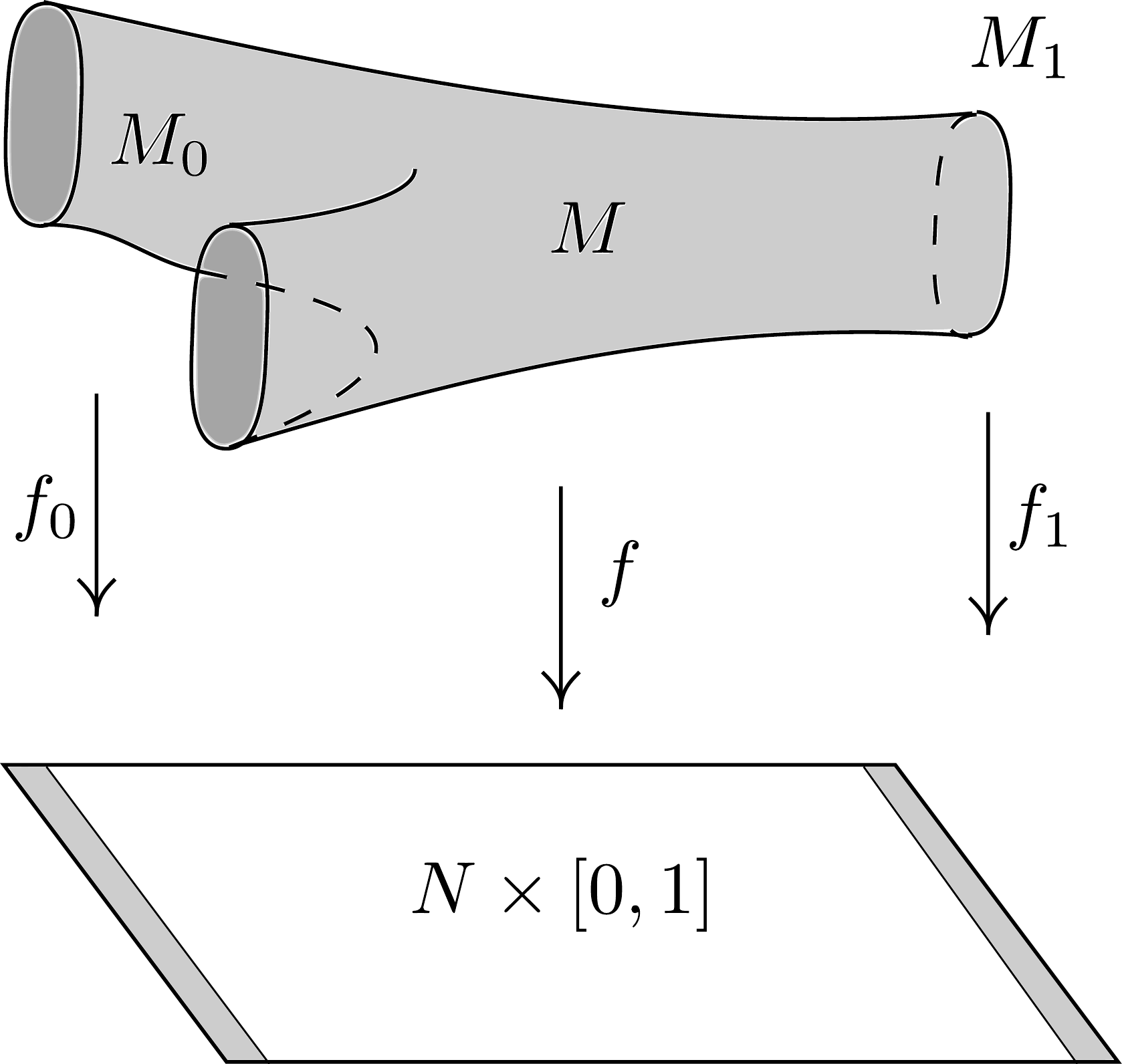}
\caption{Concordance}
\label{fig:8}
\end{minipage}%
\begin{minipage}{.5\textwidth}
  \centering
\includegraphics[height=1.4in]{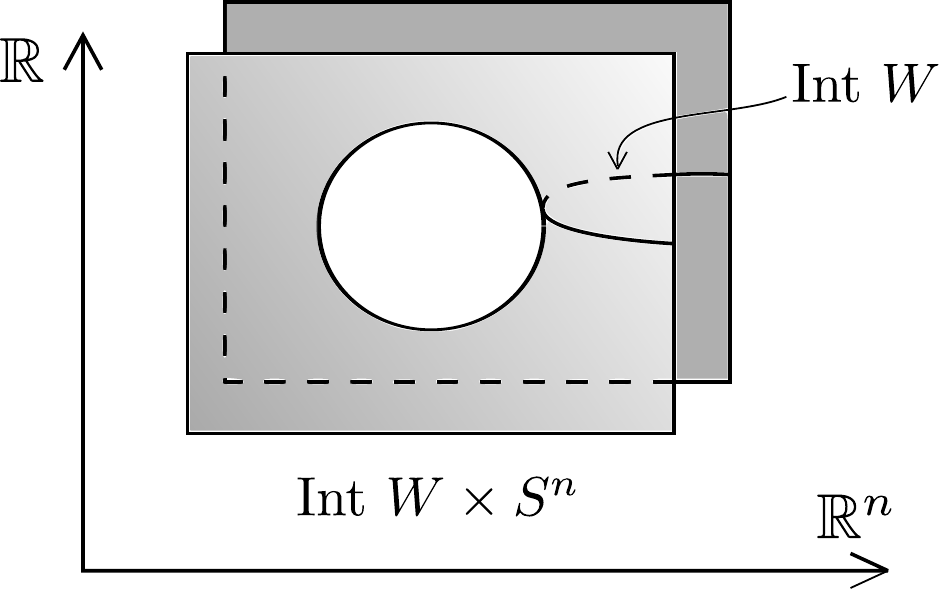}
\caption{The map $(g, \alpha)$}
\label{fig:8a}
\end{minipage}
\end{figure}

One concordance, called \emph{breaking}, is of particular interest. It is constructed by means of a compact manifold $W$ of dimension $d$, and a proper Morse function $f$ on the interior of $W$ with values in $(0, \infty)$. Suppose  that $f^{-1}[1,\infty)$ is diffeomorphic to $\partial W \times [1, \infty)$ and, furthermore, the restriction
of $f$ to the latter is the projection onto $[1,\infty)$. Then
\[
(g,\alpha): \mathop\mathrm{Int} W\times S^n \xrightarrow{f\times \mathop\mathrm{id}} (0,\infty)\times S^n \stackrel{\subset}\longrightarrow \R^{n+1}\simeq  \R^n \times \R
\]
is a broken submersion  \cite{Sa1}, see Fugure~\ref{fig:8a}. The inclusion $(0, \infty) \times S^n \subset \R^{n+1}$ in the composition takes a scalar $r$ and a vector $v \in \R^{n+1}$ of length $1$ to $rv$.
By  \cite[Proposition 4.2]{Sa1}, the map $g$ is also a fold map.

Let $i_A$ denote the inclusion of a subset $A$ into $\R$, and let $(g_A , \alpha_A)$ denote the pullback of the map $(g,\alpha): \mathop\mathrm{Int} W \times S^n \to \R^n\times \R$ with respect to 
\[
(i_A\times  {\mathop\mathrm{id}}_{\R^{n-1}})\times {\mathop\mathrm{id}}_\R: (A\times \R^{n-1})\times \R \longrightarrow \R^n \times \R.
\]
Then $(g_{[0,1]}, \alpha_{[0,1]})$ is a concordance, see Figure~\ref{fig:9}. Its inverse is  a concordance from $(g_1, \alpha_1)$ to $(g_0, \alpha_0)$. It is called the \emph{standard model for breaking concordances} as this concordance breaks fibers of a submersion.

\begin{figure}[h]
\centering
\centering
\includegraphics[height=1.5in]{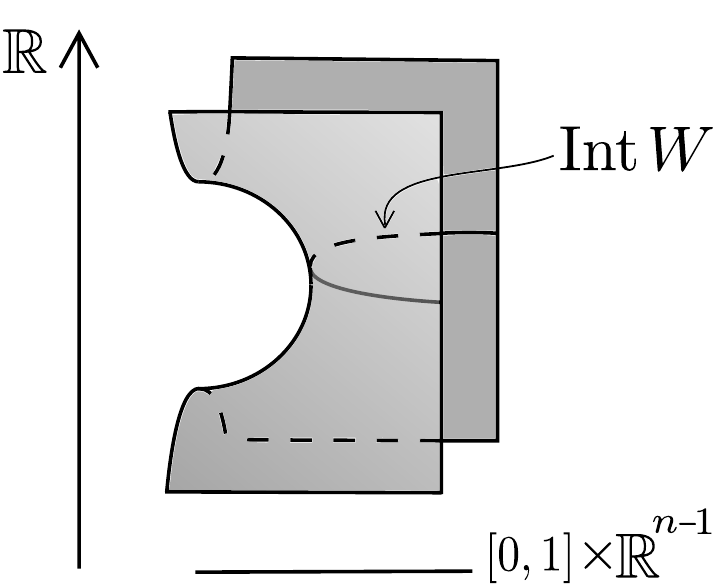}
\caption{Breaking concordance}
\label{fig:9}
\end{figure}

Finally, for any map $(f, \alpha): W\to N\times \R$ and any of its regular points $p$, there is a neighborhood $U\approx \R^{d+n-1}\times \R$ of $p$ in $W$ such that $(f, \alpha)$  has the form $(g_1, \alpha_1)$ over $U$. We say that a concordance of $(f, \alpha)$ is \emph{breaking} if it coincides with the standard model for breaking concordances over $U$, and it is trivial elsewhere (i.e., it has support in $U$).



\section{Basic concordances}\label{sec:5}

We will show that Theorem~\ref{th:4.2} follows from the Harer stability theorem. The argument is in spirit of that by Eliashberg-Galatius-Mishachev in \cite{EGM}. In this section we consider two basic concordances that will play an important role in the proof. 

\begin{example}\label{ex:2.6} Let $\pi\co E\to N$ be a fiber bundle with fiber a surface $F_g$ of genus $g$. Let  
$D_1, D_2$ be two disjoint submanifolds of $E$ such that $\pi|D_i$ is a trivial disc bundle over $N$. In particular, $D_i=N\times D^2$. We aim to construct a broken fold concordance of $\pi$ to a fiber bundle with fiber $F_{g+1}$, see Figure~\ref{fig:mum1}. 

Constant maps of $D^2\sqcup D^2$ and  $D^1\times S^1$ to a point are concordant by means of a Morse function $u\co W\to [0,1]$ with a unique critical point (of index $1$), see Fig.~\ref{fig:11}. Let $\Pi$ be the concordance of $\pi$ with support in $D_1\sqcup D_2$ and with core $\id\times u\co N\times W\to N\times [0,1]$. Then $\Pi$ is a \emph{stabilizing concordance}; it attaches to each fiber $F_g$ a handle, see Figure~\ref{fig:mum1}. 

A stabilizing concordance also exists in a slightly more general setting where $\pi\co E\to N$ is a broken fibration, and $D_1, D_2$ two disjoint submanifolds of $E$ such that each $\pi|D_i$ is a trivial disc bundle over $N$. 
\end{example}

\begin{figure}[h]
\centering
\begin{minipage}{.55\textwidth}
\centering
\includegraphics[height=1.5in]{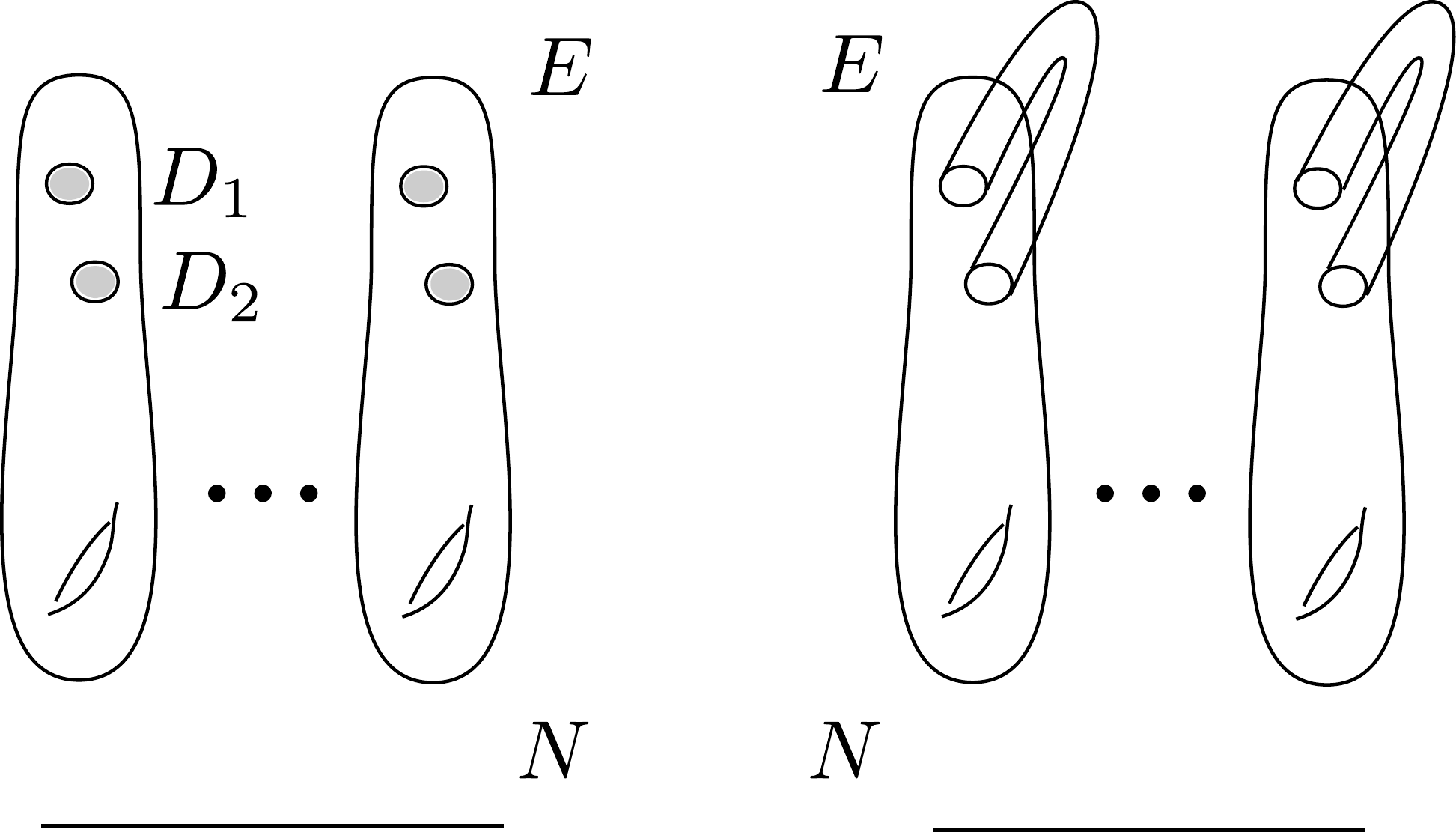}
\caption{A stabilizing concordance.}
\label{fig:mum1}
\end{minipage}%
\begin{minipage}{.55\textwidth}
  \centering
\includegraphics[height=1.5in]{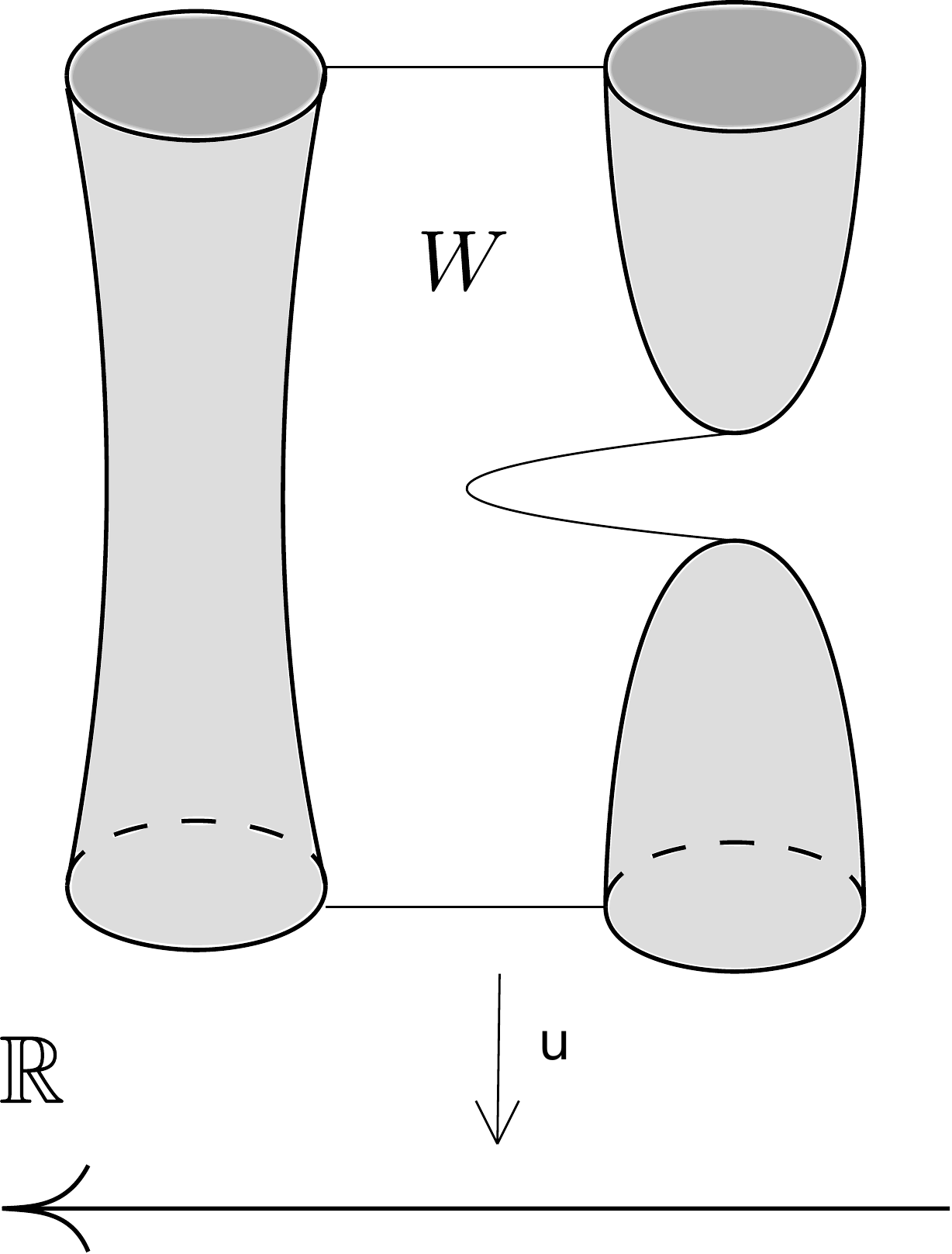}
\caption{Cobordism $W$.}
\label{fig:11}
\end{minipage}
\end{figure}


In general, however, a given fiber bundle $\pi\co E\to N$ may not contain trivial disc subbundles. For this reason we also introduce a concordance of Example~\ref{ex:2} which stabilizes fibers locally, only over a subset $U\subset N$; such a concordance always exists. First we will explain the construction in the model case where $N\subset \R^n$ is a disc and $\pi$ is a disjoint union of two disc bundles, and then we consider the general case. The fibers of this concordance are presented on Figure~\ref{fig:3}.

\begin{example}\label{ex:2}
 For the construction we will need a compact manifold $W$, and a proper Morse function $h$ on the interior of $W$ such that the fibers of $h$ over negative and positive values are $D^2\sqcup D^2$ and $D^1\times S^1$ respectively, compare $h$ with the function $u$ on  Figure~\ref{fig:11}.  
 
 Let $f_0$ be the disjoint union of two trivial disc bundles $D_i=D^2\times N\to N$, $i=1,2$, over the standard open disc $N\subset \R^n$ of radius $1$. Let $U\subset N$ be the concentric closed subdisc of radius $0.5$, see the part of Figure~\ref{fig:3} over $N\times \{0\}$. 
 
\begin{figure}[h]
\centering
\includegraphics[height=1.5in]{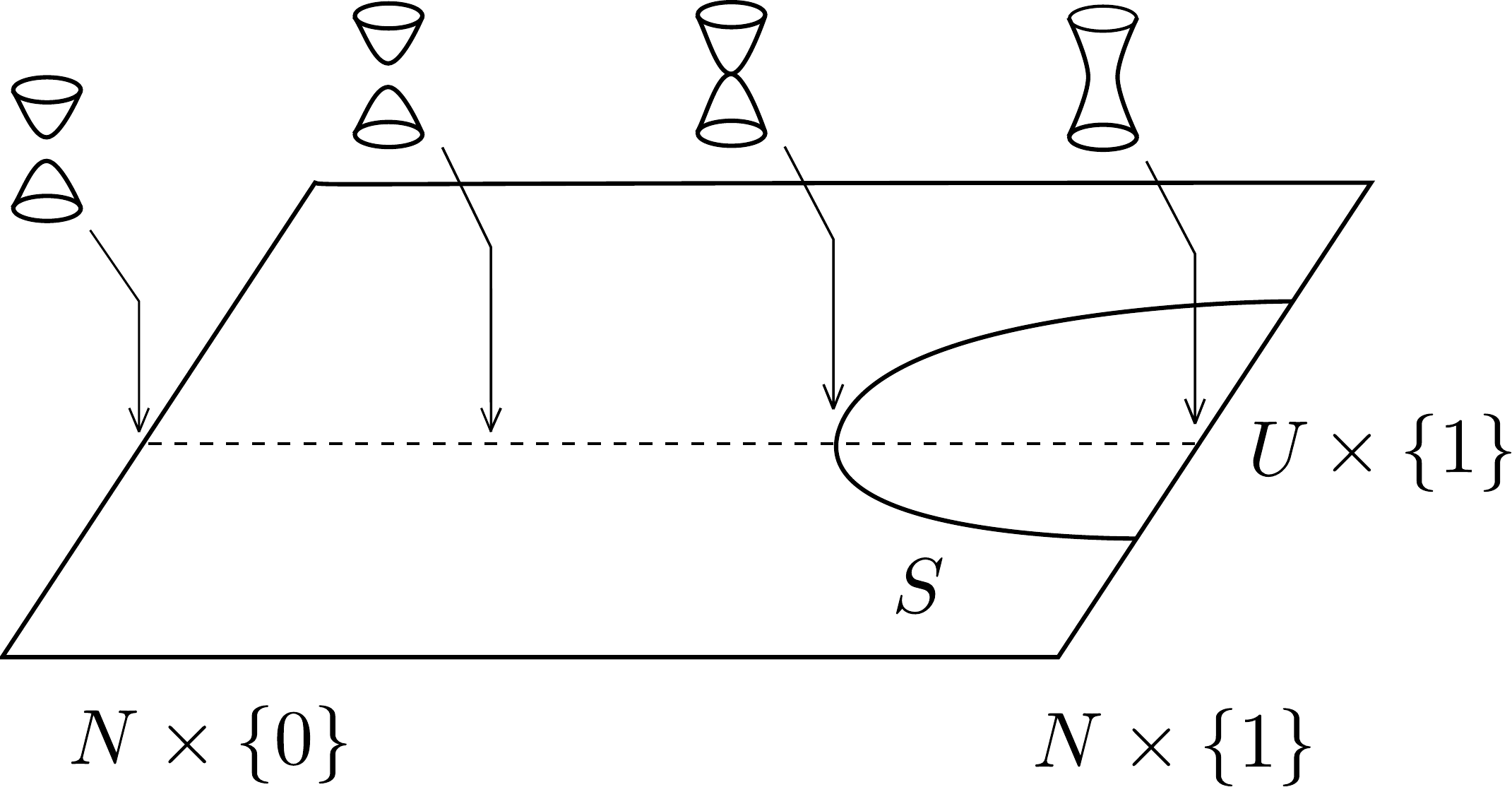}
\caption{Fibers over $N\times [0,1]$.}
\label{fig:3}
\end{figure}

 Let $S$ be the lower hemisphere of the sphere in $N\times [0,1]\subset \R^n\times \R$ of radius $0.5$ centered at $\{0\}\times \{1\}$; it meets the boundary $N\times \{1\}$ transversally along $\partial S=\partial U\times \{1\}$ and the projection of the interior of $S$ to $N\times\{1\}$ is a diffeomorphism onto the interior of $U\times \{1\}$, see Figure~\ref{fig:3}. We define $f$ to be the broken submersion to $N\times [0, 1]\subset \R^{n+1}$ given by the restriction of 
 \[
          W\times S^n \xrightarrow{h\times \id_{S^n}} \R\times S^n \longrightarrow \R^{n+1},
 \]
 where the second map in the composition takes a real number $\lambda$ and a vector $v$ of length $1$ to $\lambda v+e_{n+1}$.  Thus, over a neighborhood $S\times \R$ of $S$ in $N$ the concordance $f$ is given by $\id_S\times h$ and over each path component of the complement to $S$ it is a trivial fiber bundle. 

We will use this concordance in a more general setting.

 Let $f_0$ be a broken  submersions $E\to N$ and $U\subset N$ a small disc with smooth boundary. We aim to construct a concordance which 
 attaches to each fiber over the interior points of $U$ a handle. 
We identify $U$ with a closed ball in $\R^n$ of radius $0.5$, and a neighborhood $V$ of $U$ in $N$ with an open ball of radius $1$. If $U$ is sufficiently small, then  $E|f_0^{-1}V$ contains two disjoint submanifolds $D_1$ and $D_2$  such that each $f|D_i$ is a trivial disc bundle over $V$. 
 We have constructed the concordance of $f_0|D_1\sqcup D_2$. Since it is trivial near the fiberwise boundary, we can extend the constructed concordance trivially to a concordance of $f_0|f_0^{-1}(V)$. Since the obtained concordance is trivial near $f_0^{-1}(\partial V)$, we may extend it trivially to a desired concordance of $f_0$.  
 \end{example}

An important consequence of the concordance in Example~\ref{ex:2} is the following proposition.

\begin{proposition}\label{p:5.3c} Let $f_0\co M\to N$ be a broken submersion over a compact manifold. If over (possibly empty) $\partial N$ the original map $f_0$ is a fiber bundle with fiber $F_g$ of genus $g\gg \dim N$, then $f_0$ is concordant to a broken submersion $f_1$ with connected fibers such that each regular fiber is of genus $\gg \dim N$. 
\end{proposition}

\section{Folds of index $0$}\label{s:4a}


\subsection{Erasing concordance} Let $F$ be an oriented closed surface, and $N$ an arbitrary manifold. Then the broken submersion given by the projection $N\times F\to N$ is concordant to an empty map. The concordance is given by a broken submersion of $N\times W$ where $W$ is an oriented compact $3$-manifold with $\partial W=F$. For example, if $W$ is the standard $3$-disc of radius $1/\sqrt{2}$, then the \emph{erasing concordance} $\id_N\times h$ where $h(x)=-|x|^2+0.5$ joins the trivial sphere bundle over $N$ with the empty map. 

\subsection{Chopping concordance} Let $\pi\co E\to N$ be a submersion of dimension $2$ with fiber $F_g$ and $D\to N$ a trivial open disc subbundle of $\pi$. A \emph{chopping} concordance chops off a sphere from each fiber. More precisely, a chopping concordance modifies the fiber bundle only inside $D$ so we will assume that $E=D$. There are a bordism $W$ from $D^2$ to $D^2\sqcup S^2$, and a Morse function $f\co W\to [0,1]$ with a unique critical point of index $2$. The desired concordance is $\id_N\times f$.   \\

%
%

The following proposition at least in part appears in \cite{MW} and \cite{EGM}. 

\begin{proposition}\label{prop:6.1} Every proper broken submersion $f_0$ of even dimension $d$ to a compact simply connected manifold $N$ is concordant to a broken submersion $f_1$ with no fold points of index $0$. 
\end{proposition}
\begin{proof} Suppose that $N$ is closed. 
Let $\sigma$ be a component of folds of $f_0$ of index $0$, and let $U$ denote one of the two path components of the complement to $f_0(\sigma)$ in $N$ for which the coorientation of $f_0(\sigma)$ is outward directing. The concordance that we construct is trivial outside a neighborhood of $f^{-1}_0(\bar{U})$. Consequently, we may assume that $N$ is a neighborhood of $\bar{U}$. In fact, only  the component containting $\sigma$ is modified, and therefore, by the definition of broken submersions, we may assume that $M=\sigma\times \R^{d+1}$, and $f_0$ is the product of $\id_\sigma$ and $g=x_1^2+\cdots + x_{d+1}^2$ followed by an identification of $\sigma\times \R$ with a neighborhood of $\sigma$ in $N$. Let $S$ be a submanifold in $N\times [0,1]$ such that $\partial S=\partial \bar{U}\times \{0\}$ and the projection of the interior of $S$ to $N$ is a diffeomorphism onto $U$. Over a neighborhood $S\times \R$ of $S$ the map $f$ is given by $\id_S\times g$, while over each of the two components of the complement to $S$ in $N\times [0,1]$, the map $f$ is a trivial fiber bundle. 

Suppose now that $N$ has a non-empty boundary. Let $\sigma$ be a component of folds of index $0$. If $f_0(\sigma)$ bounds $S$ and the coorientation of $\partial S$ is outward directing, then $\sigma$ can be eliminated by the concordance of the first part of the proof.  Suppose $\partial S$ is inward directing. Let $N'$ denote the enlargement of $N$ with a collar $\partial N\times [0,1]$ attached to $N$ by means of an identification of $\partial N\subset N$ with $\partial N\times \{1\}$. Let's extend $f_0$ over the collar so that it is a concordance that first chops off a sphere from each fiber and then eliminates the choped off component by the erasing concordance. In particular the extended map $f_0$ has a new component $\sigma'$ of breaking folds of index $0$. Furthermore, the image of $\sigma'\sqcup \sigma$ bounds $S'\subset N'$ such that the coorientation of $\partial S'$ is outward directing. Hence, $\sigma$ and $\sigma'$ can be eliminated by the concordance of the first part of the proof. Thus, we can assume that $f_0$ has no folds of index $0$. 
\end{proof}

\section{Geometric consequences of the Harer stability theorem} \label{s:3}

Let $\Gamma_{g,k}$ denote the relative mapping class group of a surface $F_{g,k}$ of genus $g$ with $k$ boundary components.
There are several proofs of the Mumford conjecture, most of them use the Harer stability theorem: the homomorphism $\Gamma_{g,k}\to \Gamma_{g, k-1}$ induced by capping off a boundary component of $F_{g,k}$ and the homomorphism $\Gamma_{g,k}\to \Gamma_{g+1, k-2}$ induced by attaching a cylinder along two boundary components are homology isomorphisms  in dimensions $\ll g$.  In view of the Atiyah-Hirzebruch spectral sequence, the Harer stability theorem is equivalent to the assertion that the homomorphisms under consideration induce bordism isomorphisms of classifying spaces in dimensions $\ll g$.

\begin{example}\label{ex:3.1} By the Harer stability theorem, given a fiber bundle $f_0\co E_0\to N_0$ over a compact manifold of dimension $\ll g$ with fiber $F_{g,k}$ and a section $s$ over $\partial{N_0}$ together with a trivialization $\tau$ of the normal bundle of $s(\partial N_0)$ in $E_0|\partial N_0$, there are an oriented  bordism of $f_0$ to  $f_1\co E_1\to N_1$ and 
extensions of $s$ from $\partial N_0=\partial N_1$ over $N_1$ and $\tau$ from $s(\partial N_1)$ over $s(N_1)$.  Indeed, the initial data defines a map of pairs 
\[
(N_0, \partial N_0)\to (\mathop\mathrm{BDiff} F_{g,k}, \mathop\mathrm{BDiff} F_{g,k+1}), 
\] 
and the assertion is equivalent to the existence of a bordism to a map with image in $\mathop\mathrm{BDiff} F_{g, k+1}$. 
\end{example}

\begin{example}\label{ex:3.2} Let $f_0$ be a fiber bundle over $N_0$ with fiber $F_g$ of genus $g\gg \dim N_0$. Suppose that there exists a stabilization $f_1$ of $f_0$, see Example~\ref{ex:2.6}. Then $f_0$ is zero bordant if and only if $f_1$ is. Indeed, the assertion follows from the fact that  the two inclusions 
\[
     \mathop\mathrm{BDiff} F_{g}  \longleftarrow \mathop\mathrm{BDiff} F_{g, 1} \longrightarrow \mathop\mathrm{BDiff} F_{g+1}
\]
are bordism equivalences in stable range.
\end{example}

Eliashberg, Galatius and Mishachev gave~\cite{EGM} an important geometric interpretation of the Harer stability theorem. In this section we deduce two consequences of the Harer stability theorem (Proposition~\ref{p:6.3} and \ref{p:6.6}) for broken submersions using a singularity theory technique from \cite{EGM}. 

\begin{proposition}\label{p:6.3} Let $f_0$ be a broken submersion $M_0\to N_0$ to a closed simply connected manifold $N_0$. Then $f_0$ is bordant to a fiber bundle. 
\end{proposition}
\begin{proof} In view of Propositions~\ref{p:5.3c} and \ref{prop:6.1}, we may assume that the fiber of $f_0$ over each regular point is a connected surface of genus $\gg \dim N_0$ and that $f_0$ has no folds of index $0$. 
Let $\sigma$ denote a path component of breaking folds $\Sigma f_0$ of $f_0$. Since $f_0(\sigma)$ is cooriented and $N_0$ is simply connected, the Mayer-Vietoris 
\begin{figure}[h]
  \centering
\includegraphics[height=1.5in]{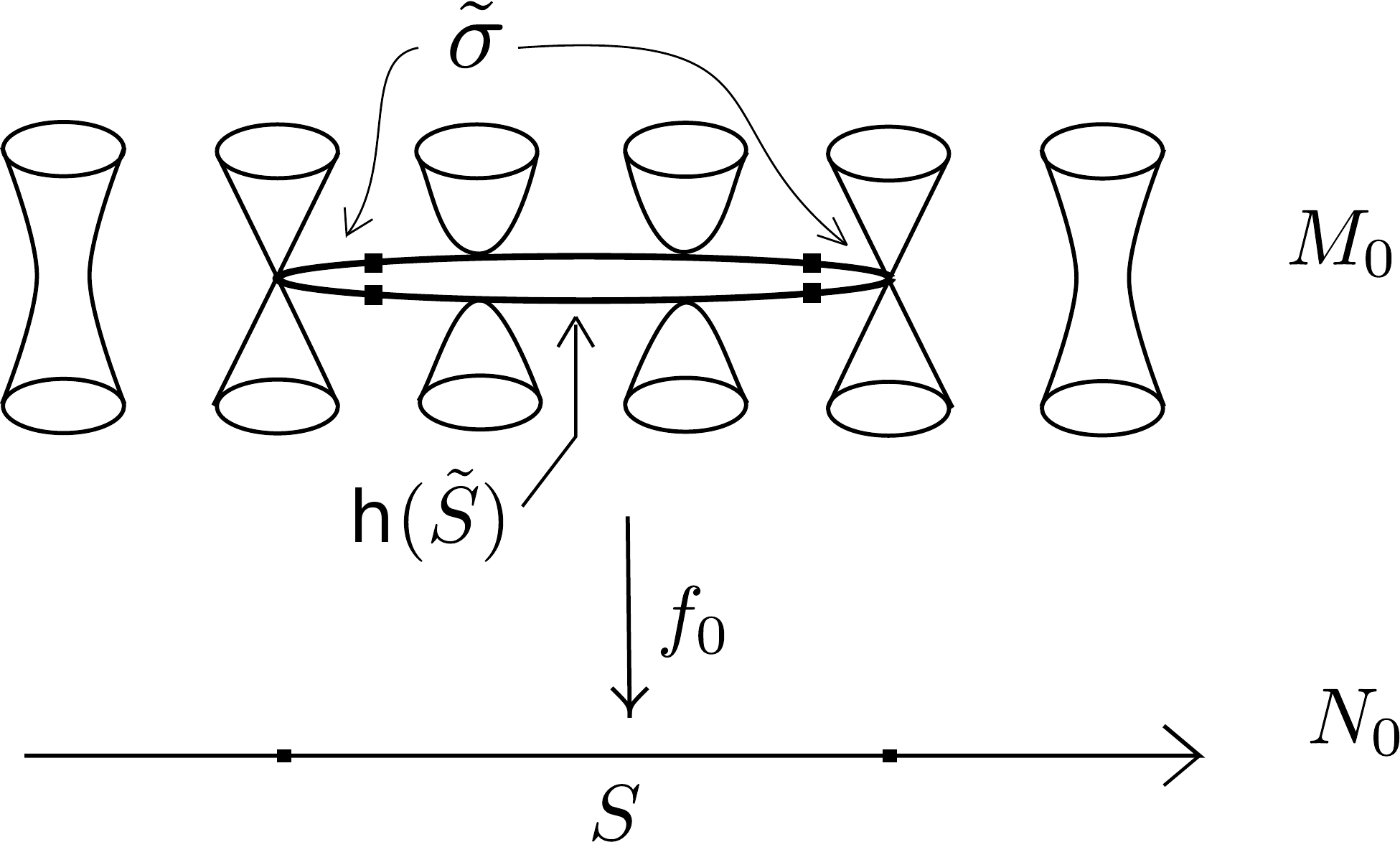}
\caption{The image $g(\tilde{S})$.}
\label{fig:4c}
\end{figure}
sequence implies that the complement to $f_0(\sigma)$ consists of two components. Let $S$ denote the closed submanifold in $N_0$ bounded by $f_0(\sigma)$ such that the coorientation of the fold values $\partial S$ is inward directed, see Figure~\ref{fig:4c}. 
Recall that  a neighborhood of $\sigma $ is identified with $\sigma\times \R^3$ and near $\sigma$ the map $f_0$ is given by $\id_\sigma\times m$ where $m=-x_1^2-x_2^2+x_3^2$. Let $\tilde\sigma$ be the submanifold $\sigma\times \{0\}\times \{0\}\times \R$ in the neighborhood of $\sigma$. Note that the coordinates $x_1$ and $x_2$ trivialize the normal bundle of $\tilde\sigma$. 
Let $\tilde{S}=S\cup_{\partial S}S$ be the double of $S$. A neighborhood $\tilde\sigma'$ of $\partial S$ in $\tilde{S}$ is canonically diffeomorphic to $\tilde\sigma'$. Given a map $h$ of $\tilde S$, the restrictions of $h$ to the two copies of $S$ are denoted by $h_+$ and $h_-$. 

In view of Lemma~\ref{l:6.4} below, we may assume that the canonical diffeomorphism $\tilde\sigma'\to \tilde\sigma$ extends to an inclusion $h\co \tilde{S}\subset M_0$ such that $h_+$ and $h_-$ are right inverses of $f_0$, and the trivialization of the normal bundle of $\tilde\sigma$ extends to that over $h(\tilde{S})$.

The promised concordance will have support in a small neighborhood $h(\tilde S)\times \R^2$ of $h(\tilde S)$; hence, we may assume that the complement is empty.  Let $S'$ be a copy of $S$ in $N_0\times [0,1]$ such that $S'$ meets the boundary of $N_0\times [0,1]$ transversally along $\partial S\times \{0\}$ and the projection of the interior of $S'$ to $N_0$ is a diffeomorphism onto the interior of $S$. Over a neighborhood $S'\times (-1,1)$ of $S'$ in $N_0\times [0,1]$, the desired concordance is $\id_{S'}\times u$, where $u$ is the Morse function of Example~\ref{ex:2.6} (see Figure~\ref{fig:11}), while over the complement  to $S'$ in $N_0$ the concordance is trivial.   
\end{proof}

\begin{lemma}\label{l:6.4} After possibly modifying $f_0$ by an oriented bordism, we may assume that there is an embedding $h\colon \tilde{S}\to M_0$ with trivialized normal bundle extending the canonical diffeomorphism $\tilde\sigma'\to \tilde\sigma$ and the trivialization of its normal bundle respectively such that $h_{-}$ and $h_+$ are right inverses to $f_0$. 
\end{lemma}
\begin{proof} 
We may assume that $f_0|\Sigma f_0$ is a general position immersion. Let $S_j$ denote the submanifold in $S$ of points of $f_0(\Sigma f_0)$ of multiplicity $j$ and $S_0$ is the complement to $\cup S_i$ in $S$. Then $S=\cup S_j$. Suppose that $h_-, h_+$ and trivializations have been constructed over a neighborhood of $S_j$ for all $j>k$. Let $D$ be an open tubular neighborhood of $\Sigma f_0$ in $M_0$. 
\begin{figure}[h]
  \centering
\includegraphics[height=1.5in]{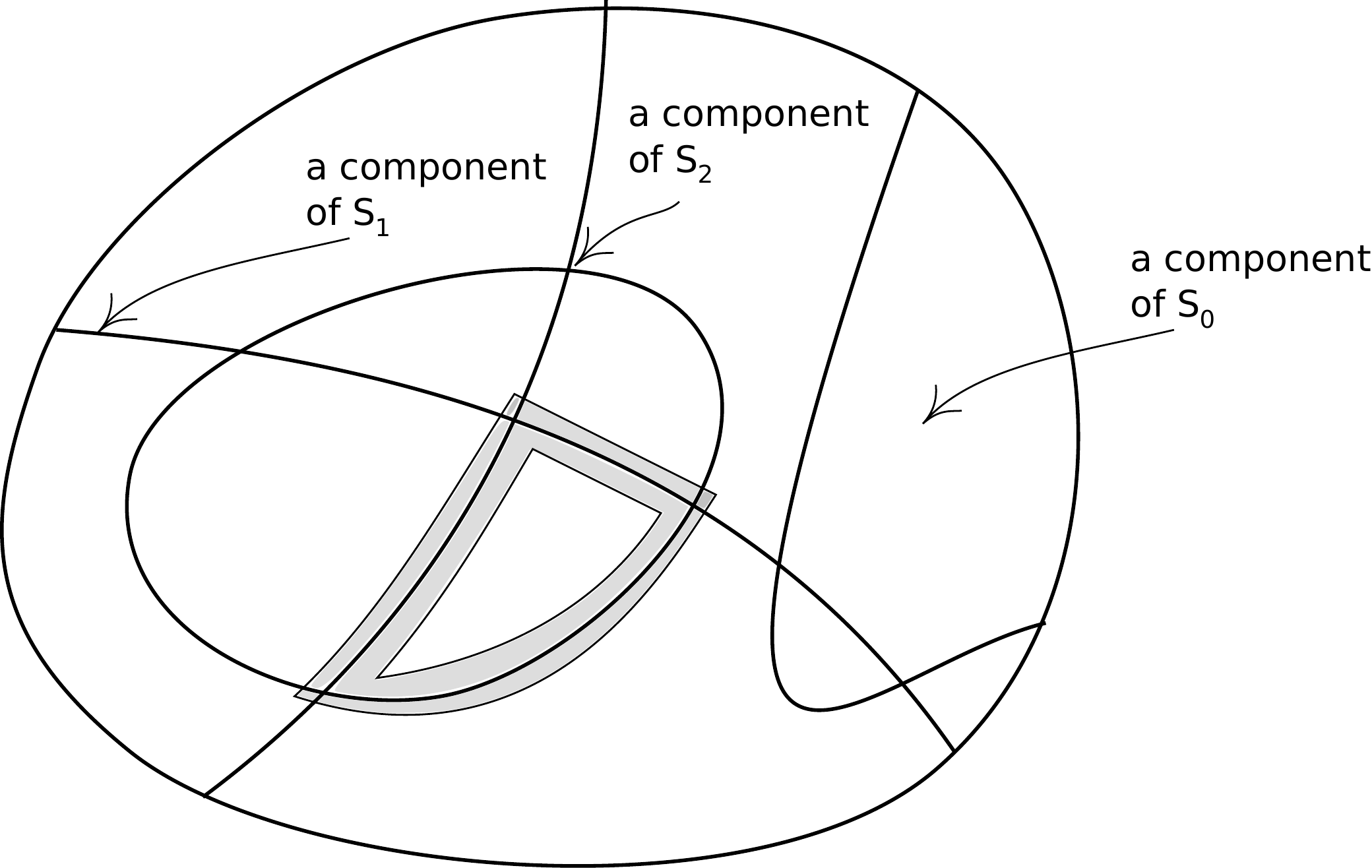}
\caption{De composition of $S$.}
\label{fig:12}
\end{figure}
Then over $B_0=S_k$ the map $b_0$ given by $f_0|M_0\setminus D$ is a fiber bundle with fiber $F_{g, 2k}$ for some $g$. By Example~\ref{ex:3.1}, there is a bordism $b\co E\to B$ of $b_0$ to $b_1\co E_1\to B_1$ such that $h_-, h_+$ and trivializations extend over $B_1$. The bordism $b$ can be essentially uniquely thickened to a bordism $\mathbf{b}\co\mathbf{E}\to \mathbf{B}=B\times D^k$ of the restriction of $f_0$ over a disc neighborhood of $B_0$ so that $\mathbf{b}$ is a broken submersion with breaking fold values $\sqcup B\times D^{k-1}_i$ where $D^{k-1}_i$ ranges over all $k$ coordinate hyperdiscs in $D^k$. 
Let $N$ be the union of $N_0\times I$ and $\mathbf{B}$ in which the top submanifold $(B_0\times D^k)\times \{1\}$ is identified with $B_0\times D^k\subset \mathbf{B}$. Let $M$ be a similar union of $M_0\times I$ and $\mathbf{E}$. Then after smoothing corners we obtain a bordism $f= f_0\times \id_I\cup \mathbf{b}$ of $f_0$ to $f_1$  such that $h_-, h_+$ and trivializations extend over a neighborhood of $S_k(f_1)$.
Thus, by induction, we get a desired extension.
\end{proof}

\begin{remark}\label{r:3.5} The above construction works in the case of $N_0=S^1$ as well. Indeed, choose $S$ to be the interval in $N_0$ over which the fibers of $f_0$ are of maximal Euler characteristic. Then the above bordism eliminates the two folds in $f^{-1}_0(\partial S)$. Continuing by induction we end up with a submersion. Note that here the bordism of $f_0$ is actually a concordance.  
\end{remark}

\begin{figure}
\includegraphics[height=1.5in]{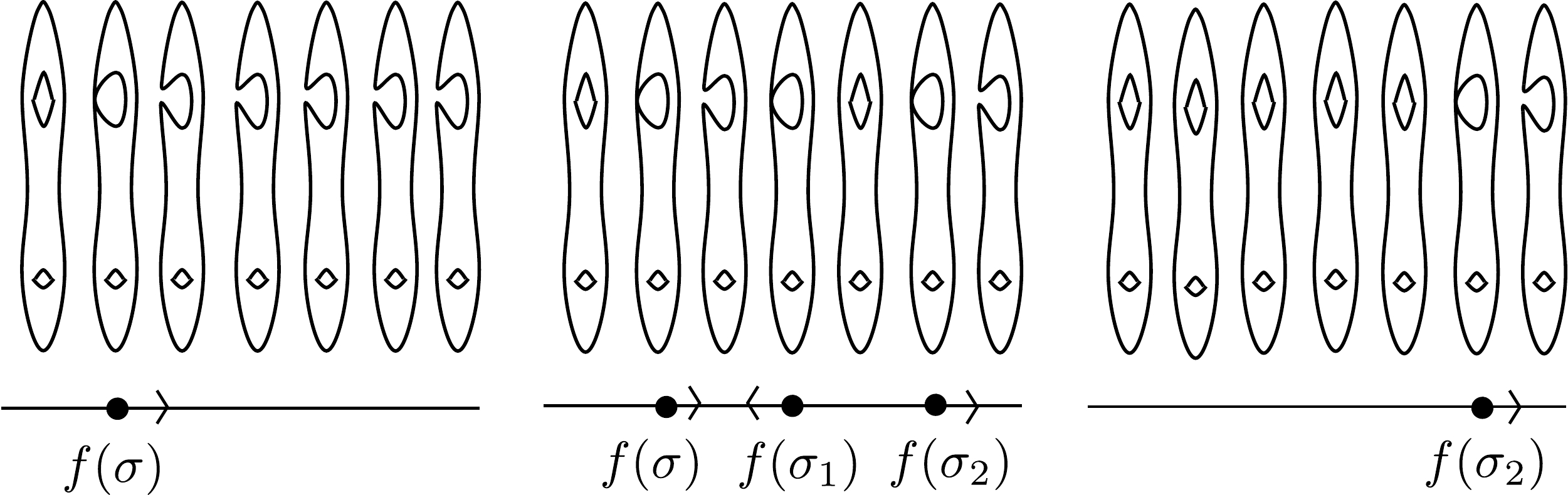}
\caption{Trading singularities.}
\label{fig:6}
\end{figure}

\begin{proposition}\label{p:6.6}
Let $f_0$ be a broken submersion $M_0\to N_0$ to a compact  simply connected manifold $N_0$. Suppose that over $\partial N_0$  the map $f_0$ is a fiber bundle with fiber $F_g$ of genus $g\gg \dim N_0$. 
Then $f_0|\partial N_0$ is zero bordant in the class of fiber bundles.
\end{proposition}
\begin{proof}
In view of Propositions~\ref{p:5.3c} and \ref{prop:6.1}, we may assume that the fiber of $f_0$ over each regular point is a connected surface of genus $\gg \dim N_0$ and that $f_0$ has no folds of index $0$. 
 Let $\sigma$ be a component of folds, and $S$ a closed domain bounded by $f_0(\sigma)$. 
Assume that the coorientation of $\partial S$ is outward directing; otherwise $\sigma$ can be eliminated as above. We may assume that a neighborhood of $\partial N_0$ is identified with $\partial N_0\times [0,2)$ and over $U=\partial N_0\times [0,1]$ the broken submersion $f_0$ is the trivial concordance of $f_0|\partial N_0$. Modify $f_0$
over $U$  so that it is a concordance that first stabilizes the fibers and then destabilizes them back, see Example~\ref{ex:2.6}. Then $f_0$ has two new components of breaking folds. One of these two components can be eliminated with $\sigma$ by the concordance as above. Thus the component $\sigma$ can be ``traded" for a new component of breaking folds parallel to $\partial N_0$. Consequently, we may assume that $f_0$ only has breaking folds parallel to $\partial N_0$. 
\begin{figure}
\centering
\centering
\includegraphics[height=1.5in]{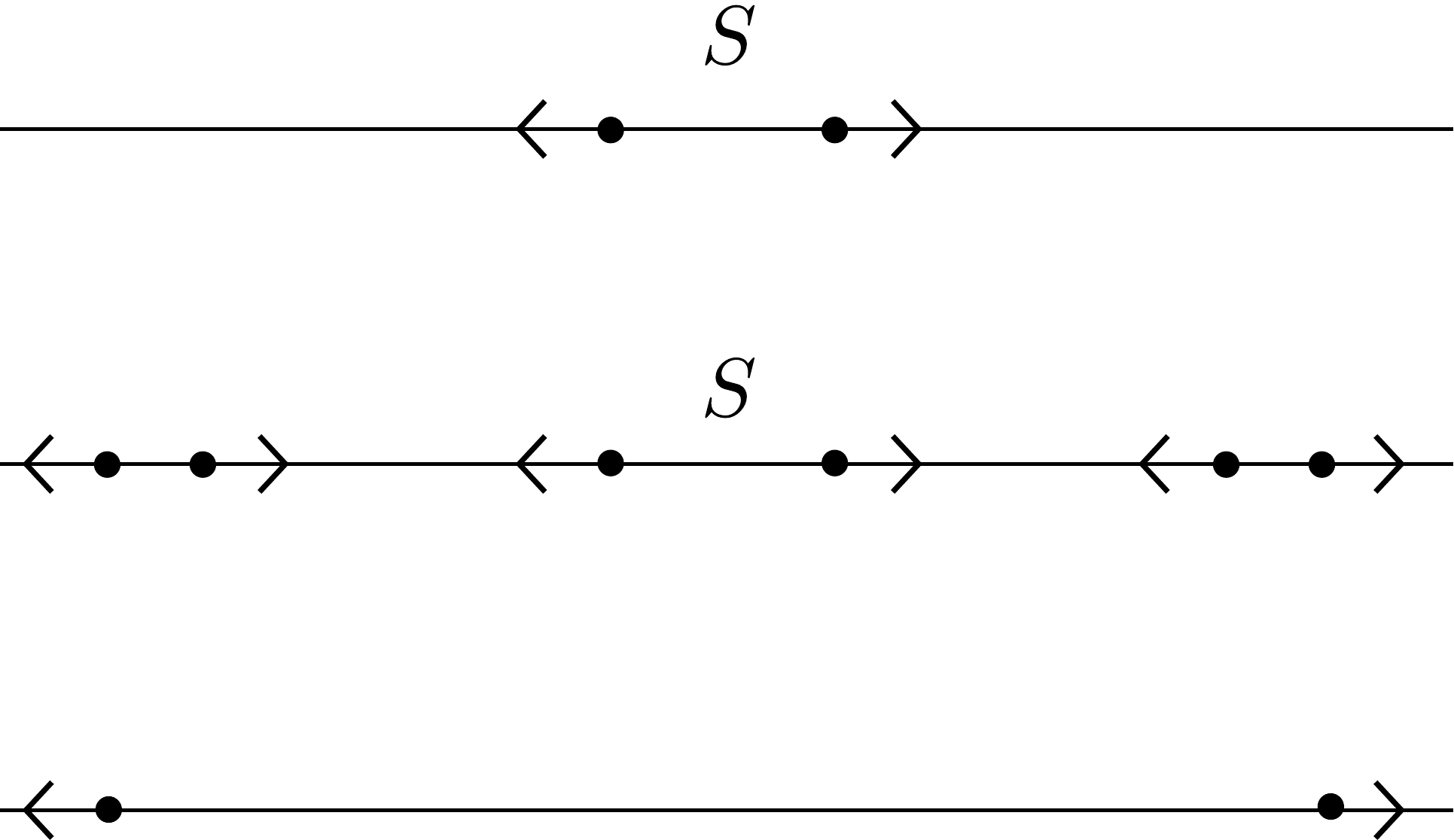}
\caption{The component $\sigma$ can be ``traded" for a new component of breaking folds parallel to $\partial N_0$.}
\label{fig:7}
\end{figure}

In other words, the map $f_0$ over a collar neighborhood of $\partial N_0$ is a concordance that stabilizes the fibers, and over the complement to the collar neighborhood of $\partial N_0$  it is a fiber bundle. It remains to apply  Example~\ref{ex:3.2}. 
\end{proof}

\section{The weak b-principle}\label{s:5a}

A collection $\mathcal{C}$ of smooth maps $f\co M\to N$ with fixed $\dim M-\dim N=d$ is said to be a \emph{class} of maps of dimension $d$ if the induced map  $h^*g$ in the pullback diagram 
\[
\begin{CD}
       M' @>>> M \\
       @Vh^*gVV @VgVV \\
       N' @>h>> N. 
\end{CD}
\] 
is in $\mathcal{C}$ for every map $g\in \mathcal{C}$ and every map $h$ transverse to $g$.

\begin{example} If $g\co M\to N$ is a submersion, then for every smooth map $g\co N'\to N$, the induced map  $h^*g$ 
is a submersion as well. If $g$ is an immersion, then the induced map $h^*g$ is an immersion as well provided that $h$ is \emph{transverse} to $g$, i.e., provided that for each $x\in N$,  $x'\in N'$ and $y\in M$ such that $h(x')=x=g(y)$, we have
\[
   \mathop\mathrm{Im}(d_{x'}h)\ \oplus\  \mathop\mathrm{Im}(d_yg)\ \simeq \ T_xN.
\] 
Thus, both submersions and immersions of dimension $d$ form classes of maps. More generally, solutions to any open stable differential relation $\mathcal{R}$ form a class of maps \cite{Sa}.  The transversality condition is clearly important here:  if a smooth map $h$ is not transverse to a smooth map $g$, then the pullback space $M'$ may not admit a manifold structure. 
\end{example}

An appropriate quotient space of all proper maps in a collection $\mathcal{C}$ is called the \emph{moduli space} for $\mathcal{C}$. 
Namely, recall that the \emph{opening} of a subset $X$ of a manifold $V$ is an arbitrarily small but non-specified open neighborhood $\mathop\mathrm{Op}(X)$ of $X$ in $V$. Consider the affine subspace \[ \{x_1+\cdots + x_{m+1}=1\}  \subset \R^{m+1}. \] It contains the standard simplex $\Delta^m$ bounded by all additional conditions $0\le x_i\le 1$. Let $\Delta^n_e$ denote the opening of $\Delta^m$ in the considered affine subspace. Then every morphism $\delta$ in the simplicial category extends linearly to a map $\tilde\delta\co \Delta^m_e\to \Delta^n_e$. 
Let $X_m$ denote the subset of $\mathcal{C}$ of proper maps to $\Delta^m_e$ transverse to all extended face maps. Then $X_\bullet$ is a simplicial set with structure maps $X(\delta)$ given by the pullbacks $f\mapsto \tilde\delta^*f$. 


The (simplicial model of the) \emph{moduli space} $\M$ for $\mathcal{C}$ is the semi-simplicial geometric realization of $X_\bullet$. We say that $\mathcal{C}$ satisfies the \emph{sheaf property} if $f\co M\to N$ belongs to $\mathcal{C}$ whenever each $f|f^{-1}U_i$ is in $\mathcal{C}$ for a covering $\{U_i\}$ of $N$. If $f$ satisfies the sheaf property, then 
the sets $\Omega_*\M$ and $[N, \M]$ are isomorphic to the sets of bordism classes and concordance classes of proper maps in $\mathcal{C}$ to $N$ respectively. 



\begin{figure}[h]
 \centering
\includegraphics[height=1.5in]{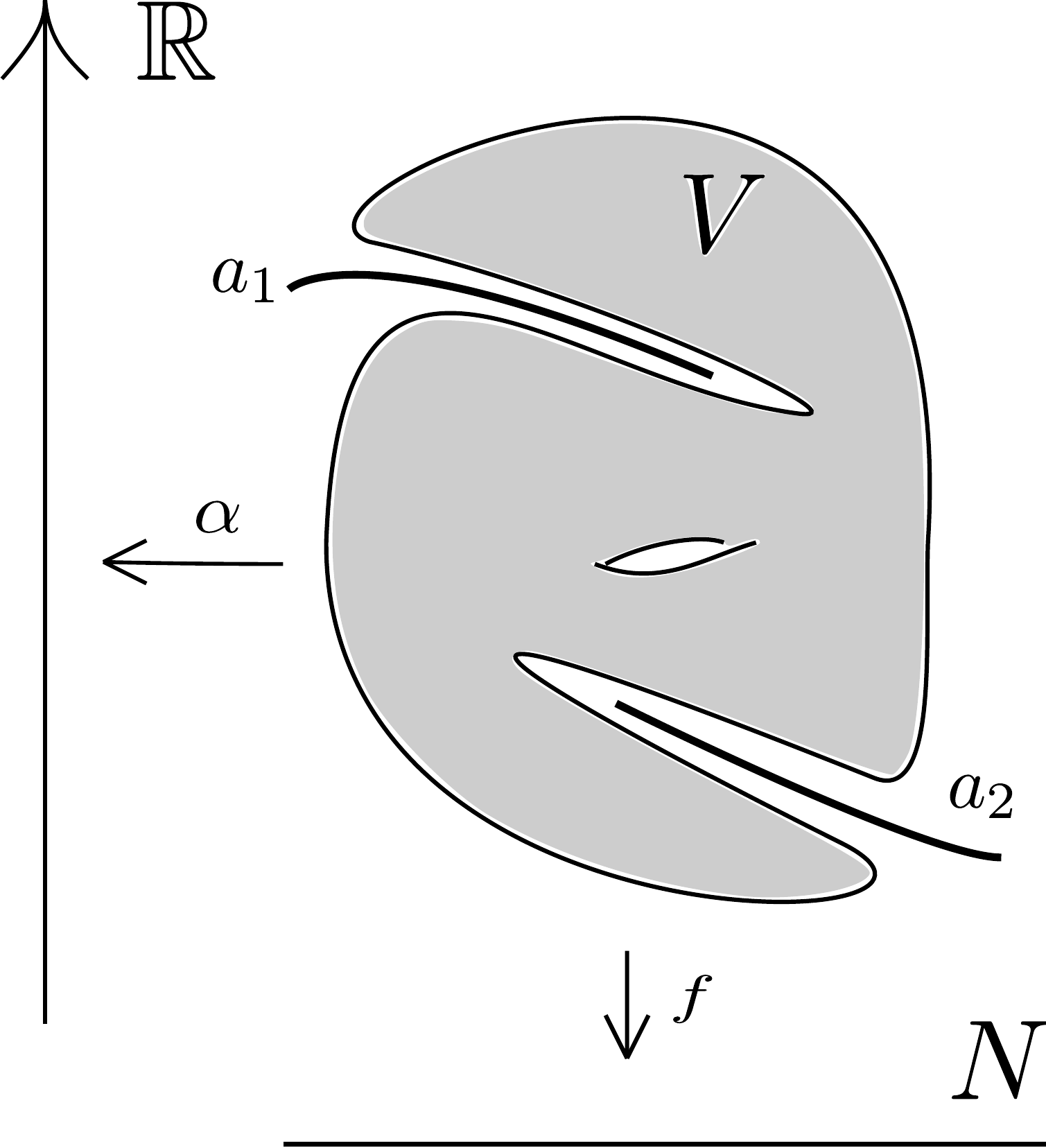}
\caption{A map $(f, \alpha)$ in the collection $h\mathcal{C}^1$.}
\label{fig:2}
\end{figure}

 We say that a class $\mathcal{C}$ is \emph{monoidal} if the map of the empty set to a point is a map in $\mathcal{C}$ and the class $\mathcal{C}$ is closed with respect to taking disjoint unions of maps, i.e., if $f_1: M_1\to N$ and $f_2: M_2 \to N$ are maps in $\mathcal{C}$, then 
\[
f_1\sqcup f_2: M_1\sqcup M_2 \longrightarrow N
\]
 is also a map in $\mathcal{C}$. For a monoidal class $\mathcal{C}$ the space $\mathcal{M}$ is  an H-space with a coherent operation (i.e., the first term of a $\Gamma$-space). We will recall the construction of its classifying space $\mathcal{M}^1$ and an approximation of $\mathcal{M}^1$ by a space $h\mathcal{M}^1$ of a relatively simpler homotopy type,  for details see \cite{Sa}, \cite{Sa1}. 
 
Let $\mathcal{C}^1$ be the derived collection (not a class) of proper maps $(f, \alpha)\co V\to N\times \R$ with $f\in \mathcal{C}$ such that every regular fiber of $(f, \alpha)$ is null-cobordant; and let $\mathcal{C}^1\subset h\mathcal{C}^1$ be a subcollection of pairs with $\alpha\circ f^{-1}(x)\ne \R$ for all $x\in N$. The  spaces $\M^1$ and $h\M^1$ are the geometric realizations of simplicial sets of maps $(f, \alpha)$ to $\Delta^m_e\times \R$ such that $f$ is transverse to all extended face maps and $(f, \alpha)$ is in $\mathcal{C}^1$ and $h\mathcal{C}^1$ respectively.

\begin{definition} The \emph{weak b-principle} for $\mathcal{C}$ is said to hold true  if the inclusion $\M^1\to h\M^1$ is a homotopy equivalence.
\end{definition}
 
\begin{theorem}[Sadykov, \cite{Sa1}]\label{th:2}  Let $\mathcal{C}$ be a monoidal class of maps satisfying the sheaf property. Suppose that every breaking concordance of every map in $h\mathcal{C}^1$ is itself in $h\mathcal{C}^1$. Then the weak b-principle for $\mathcal{C}$ holds true.
\end{theorem}

Under the assumptions of Theorem~\ref{th:2}, if $\mathcal{M}$ is path connected, then it is homotopy equivalent to its group completion $\Omega\mathcal{M}^1$. Furthermore, in view of Theorem~\ref{th:2}, we can identify $\mathcal{M}$ with $\Omega h\mathcal{M}^1$.

\section{Colored broken submersions}\label{s:5}

A map $f\co M\to N$ may not be a broken submersion even if its restriction to every subset $f^{-1}(U_i)$ for an open covering $\{U_i\}$ of $N$ is a broken submersion. In other words, broken submersions do not satisfy the \emph{sheaf property}. We will use colored broken submersions that satisfy the sheaf property.  

Let $\mathcal{I}$ denote the category of finite sets $\mathbf{n}=\{1,...,n\}$ for $n\ge 0$ and injective maps. It is a symmetric monoidal category with operation given by taking the disjoint union $\mathbf{m}\sqcup \mathbf{n}$ of objects in $\mathcal{I}$. An \emph{$\mathbf{m}$-coloring} on  a broken submersion $f$ is a map $C_f$ from the set of path components of breaking folds of $f$ to the set $\mathbf{m}$ such that the restriction of $f$ to breaking components of any fixed color is an embedding; here we allow $\mathbf{m}$ to be any element in $\mathcal{I}$ or the set $\infty$ of positive integers.  The moduli space of $\mathbf{m}$-colored broken submersions is denoted by $\M_\mathbf{m}$. Recall that an $\mathcal{I}$-space is a functor $\mathcal{I}\to \mathbf{Top}$. We are interested in the $\mathcal{I}$-space $\mathbf{m}\mapsto \M_{\mathbf{m}}$; its hocolim  is denoted by $\M_b$, see \cite{Sch}.

\begin{theorem}\label{th:5.1} The set of oriented bordism classes of broken submersions of dimension $2$ over closed oriented manifolds of dimension $n$  is naturally isomorphic to $\Omega_n( \M_b)$.
\end{theorem}
\begin{proof} Given a broken submersion $f$ over an oriented closed manifold, a choice of a coloring on its folds determines a class $\tau(f)$ in $\Omega_*(\M_b)$. We may choose a coloring so that different breaking components are colored by different colors. Then, since every isomorphism $\mathbf{m}\to \mathbf{m}$ is a morphism in $\mathcal{I}$, the class $\tau(f)$ does not depend on the choice of the coloring. If $f$ is bordant to a broken submersion $g$, then we may assume that the images of the classifying maps of $f$ and $g$ are in $\M_{\mathbf{m}}$ for a sufficiently big palette $\mathbf{m}$ and therefore $\tau(f)=\tau(g)$. 
Conversely, every map $\tau\co N\to \M_b$ representing a bordism class in $\Omega_*(\M_b)$ is linearly homotopic to a map with image in $\M_{\mathbf{m}}$ for some sufficiently big palette $\mathbf{m}$, and therefore every map $\tau$ determines a colored broken submersion. 
\end{proof}

The same argument shows that the canonical map of the telescope $\mathcal{M}_\infty=\colim\M_{\mathbf{m}}$ to $\M_b$ and the canonical map $\mathcal{M}_b\to \M_\infty$ are homotopy equivalences. In particular, homotopy classes $[N, \M_b]$ are in bijective correspondence with concordance classes of $\infty$-colored broken maps to $N$. Similarly, the homotopy colimit of the $\mathcal{I}$-space $\mathbf{m}\mapsto \M_{\mathbf{m}}^1$ is denoted by $\M_b^1$ and $\colim \M_{\mathbf{m}}^1\simeq \M_b^1$. 

A general argument on $\mathcal{I}$-spaces shows that $\M_b$ is an infinite loop space, see \cite{Sch}. Alternatively, the Galatius-Madsen-Tillmann-Weiss argument in \cite{Sa} shows that $\M_b$ is an infinite loop space, and its classifying space is $\M_b^1$. The H-space operation on $\M_b$ is defined by 
\[
    \M_\mathbf{m}\times \M_\mathbf{n} \longrightarrow \M_{\mathbf{m}\sqcup \mathbf{n}},
\]
\[
     \Delta_f\times \Delta_g \mapsto \Delta_{f\sqcup g},
\]
where $\Delta_h$ is the simplex in the moduli space corresponding to a map $h$.  We choose the unit point to be 
the vertex in $\M_{\emptyset}\subset \M_b$ corresponding to the map $\emptyset\to \Delta^0_e$.  

Since $\M_b$ is path connected, we have $\M_b\simeq \Omega\M_b^1$. Furthermore, by Theorem~\ref{th:2} the weak b-principle for colored broken submersions holds true. Consequently, $\M_b\simeq \Omega h\M_b^1$.

\section{Proof of the Mumford conjecture}\label{s:7c}

\begin{proof}[Proof of Theorem~\ref{th:main1}]
Let $h\M\simeq \Omega^{\infty}\MTSO(2)$ be the moduli space for oriented stable formal submersions of dimension $2$. We need to show that the map $\BDiff F_g\to h\M$ induces an isomorphism of homology groups in dimensions $\ll g$. 
Recall that $h\M^1$ is the geometric realization of the simplicial set whose simplicies are given by pairs of proper maps $(f, \alpha)$ to $\Delta^n_e\times \R$ such that $f$ is a submersion of dimension $2$, see \cite{Sa}. 
The simplicies of a bigger simplicial complex $h\M_b^1$ correspond to proper maps $(f, \alpha)$ to $\Delta^n_e\times \R$ such that $f$ is a broken submersion of dimension $2$ whose components of folds are labeled. 
Hence, there is an inclusion $h\M^1\to h\M^1_b$, which defines a map of the loop space $h\M\simeq \Omega h\M^1$ to the loop space $\Omega h\M_b^1\simeq \M_b$. Hence, we get a sequence of maps 
\[
      \eta\co \BDiff F_g\longrightarrow h\M\longrightarrow \M_b.
\]
Since $\M_b$ is an H-space, its fundamental group is abelian and therefore equals $[S^1, \M_b]$. On the other hand, every broken submersion over $S^1$ is concordant to a fiber bundle with fiber $F_g$, see Remark~\ref{r:3.5}. Hence, the fundamental group of $\M_b$ is the image of the perfect group $\pi_1(\BDiff F_g)$ provided that $g\ge 3$. Consequently, the space $\M_b$ is simply connected. In particular, every bordism class of $\M_b$ is represented by a map of a simply connected manifold $N$. By Proposition~\ref{p:6.3}, every broken submersion over a closed simply connected manifold $N$ is bordant to a fiber bundle with fiber $F_g$. Thus, $\eta$ induces an epimorphism in integral homology groups in dimensions $n\ll g$.  

Let us show that $\eta_*$ is injective in dimensions $n\ll g$, i.e., given a broken submersion $f_0$ over $N_0$ which restricts over $\partial N_0$ to a fiber bundle with fiber $F_g$ of genus $g\gg \dim N_0$, there is a fiber bundle $f_1$ over $N_1$ that restricts over $\partial N_1=\partial N_0$ to $f_0|\partial N_0$. Again, we may assume that $N$ is simply connected. Thus, the statement follows from Proposition~\ref{p:6.6}.  
This implies that $\eta_*$ is an isomorphism in integral homology groups in dimensions $\ll g$. Consequently,  the b-principle map $\BDiff F_g\to h\M$ induces an injective homomorphism in homology groups in a stable range. On the other hand, by the Miller-Morita theorem, the induced homomorphism in rational homology groups is also surjective in a stable range \cite{Mo}. This implies the Mumford conjecture. 
\end{proof}

 \begin{proof}[Proof of Theorem~\ref{th:4.2}]  We may turn the map $\eta\co \BDiff F_g\to \M_b$ defined in the proof of Theorem~\ref{th:main1} into a cofibration. Then the pair $(\M_b, \BDiff F_g)$ classifies bordism classes 
 \[
 (f, \partial f)\co (M, \partial M)\longrightarrow (N, \partial N)
 \]
 such that $f$ is a smooth broken submersion over $N$ that restricts over the boundary $\partial N$ to a fiber bundle $\partial f$ with fiber $F_g$, $\dim N\ll g$. It remains to observe that $\Omega_*(\M_b, \BDiff F_g)=0$ for $*\ll g$ since $\eta_*$ is an isomorphism in a stable range. 
 \end{proof}

\end{document}